\documentclass[12pt]{article}
\usepackage{amsmath}
\usepackage[affil-it]{authblk}
\usepackage{amsfonts}
\usepackage{amsthm}
\usepackage{amssymb}
\usepackage{makecell}
\usepackage{appendix}
\usepackage{algorithmic}
\usepackage{graphicx}
\usepackage{algorithm2e}
\usepackage{tabu}
\usepackage{color}
\definecolor{darkgreen}{rgb}{0.0, 0.63, 0.0}
\theoremstyle{definition}

\usepackage{mathtools}

\theoremstyle{lemma}
\newtheorem{theorem}{Theorem}

\usepackage[margin=0.5in]{geometry}

\begin{document}
    \title{Strongly Reduced Lattice Bases}
    \author{Christian Porter\footnote{Imperial College London, United Kingdom. Corresponding author, c.porter17@imperial.ac.uk}}
    \maketitle
\begin{abstract}
    In this paper, we show that for each lattice basis, there exists an equivalent basis which we describe as ``strongly reduced''. We show that bases reduced in this manner exhibit rather ``short'' basis vectors, that is, the length of the $i$th basis vector of a strongly reduced basis is upper bounded by a polynomial factor in $i$ multiplied by the $i$th successive minima of the lattice. The polynomial factor seems to be smaller than other known factors in literature, such as HKZ and Minkowski reduced bases. Finally, we show that such bases also exhibit relatively small orthogonality defects.
\end{abstract}
\section{Introduction}
A lattice $\Lambda$ is a discrete subgroup of $\mathbb{R}^m$ for some positive integer $m$. A lattice can be represented by the $\mathbb{Z}$-span of a basis $B=\{\mathbf{b}_1,\dots,\mathbf{b}_n\}$ for some $n \leq m$ (we say that $\Lambda$ is full-rank if $n=m$), and is said to be of rank $n$ if the vectors composing $B$ are linearly independent over $\mathbb{R}$. It is known that multiplying $B$ by any element of $GL_n(\mathbb{Z})$ (if $n=m$) results in another basis $B^\prime$ for $\Lambda$, and so there exists infinitely many basis representations of $\Lambda$ when $n>1$. We define the $i$th successive minimum of $\Lambda$ by
\begin{align*}
    \lambda_i=\min_{r \in \mathbb{R}^+}\{r: \mathcal{B}_n(0,r) \hspace{1.5mm} \text{contains at least $i$ linearly independent vectors of $\Lambda$}\},
\end{align*}
where $\mathcal{B}_n(0,r)$ is the $n$-dimensional ball of radius $r$ centred at the origin of $\mathbb{R}^m$.

The ``quality'' of a lattice basis is generally determined by how ``close'' the basis vectors are to the successive minima of the lattice - a ``bad'' basis will have basis vectors whose lengths are much longer than the successive minima of the lattice, whilst a ``good'' basis contains vectors whose vector lengths are relatively close to the successive minima. The study of constructing ``good'' bases for lattices is known as \emph{reduction theory}. Examples of reduced bases include those in the sense of Minkowski \cite{geometriederzahlen}, Korkin-Zolotarev \cite{surlesformesquadratiques}, Lenstra-Lenstra-Lovasz \cite{LLL} and Venkov \cite{venkov}.

Let $B=\{\mathbf{b}_1,\dots,\mathbf{b}_n\}$ be a basis for a lattice $\Lambda$. We define by
\begin{align*}
    &\mathbf{b}_i(j)=\mathbf{b}_i-\sum_{k=1}^{j-1}\mu_{i,k}\mathbf{b}_k(k), \hspace{2mm} \forall 1\leq j \leq i \leq n,
    \\&\mathbf{b}_1(1)=\mathbf{b}_1,
    \\& \mu_{i,j}=\frac{\langle \mathbf{b}_i,\mathbf{b}_j(j)\rangle}{\|\mathbf{b}_j(j)\|^2}, \hspace{2mm} \forall 1 \leq j <i \leq n.
\end{align*}
Then for any lattice vector $\mathbf{v}=\sum_{i=1}^n x_i \mathbf{b}_i$, $x_i \in \mathbb{Z}$, it holds that
\begin{align*}
    \|\mathbf{v}\|^2=\sum_{i=1}^n\left(x_i+\sum_{j=i+1}^nx_j\mu_{i,j}\right)^2\|\mathbf{b}_i(i)\|^2,
\end{align*}
where here $\|\cdot\|$ denotes the regular Euclidean norm and $\sum_{i=k}^l=0$ if $k>l$.

Another measure of the quality of a basis is by its \emph{orthogonality defect}, which is defined by
\begin{align*}
    \Delta(B)\triangleq \prod_{i=1}^n\frac{\|\mathbf{b}_i\|^2}{\|\mathbf{b}_i(i)\|^2},
\end{align*}
which roughly measures how close to being completely orthogonal a lattice basis is.
\subsection{Main Results}
In this paper, we define a new method of reduction which we call \emph{strong reduction}. We ascertain two main results about strongly reduced bases.
\begin{theorem}\label{successiveminima}
    Let $B=\{\mathbf{b}_1,\dots,\mathbf{b}_n\}$ be a strongly reduced basis of a lattice $\Lambda$ with successive minima $\lambda_1,\dots,\lambda_n$. Then for all $1 \leq i \leq n$, we have $\|\mathbf{b}_i\|^2 \leq \max\left\{1,\frac{i-k_i}{4}+\frac{k_i}{16}\right\}\lambda_i^2$, where $k_i$ is the number of elements $j$ in the set $\{1,\dots,i-1\}$ such that $\|\mathbf{b}_j\|>\lambda_j$.
\end{theorem}
\begin{theorem}\label{orthogonalitydefect}
    If $B$ is strongly reduced, then
    \begin{align*}
        \Delta(B) \leq \gamma_n^n\max_{\alpha_n}\prod_{j=\alpha_n}^n\left(\frac{n-\alpha_n+1}{4}+\frac{i}{16}\right),
    \end{align*}
    where $\alpha_n \in \{\lfloor \beta_n\rfloor, \lceil \beta_n \rceil\}$, $\gamma_n \leq \frac{n}{8}+\frac{6}{5}$ denotes the Hermite constant of rank $n$ and $\beta_n$ is the positive real root of the polynomial $4\sigma(4\sigma+1)(4\sigma+2)(4\sigma+3)- 16(n+3\sigma+1)(n+3\sigma+2)(n+3\sigma+3)=0$, where $\sigma=n-k$.
\end{theorem}
\section{Strongly Reduced Lattice Bases}
Let $\Lambda$ be a lattice with a basis $B=\{\mathbf{b}_1,\dots,\mathbf{b}_n\}$. We will assume without loss of generality that $\Lambda$ is full-rank. We say that $B$ is \emph{strongly reduced} if it satisfies the following:
\begin{enumerate}
    \item For all $1 \leq i \leq n$, let $\mathbf{v}_i=\sum_{j=1}^n x_j^{(i)}\mathbf{b}_j$ for some $x_j^{(i)} \in \mathbb{Z}$ such that $\mathbf{v}_1,\dots,\mathbf{v}_n$ are linearly independent and $\|\mathbf{v}_i\|=\lambda_i$. Then $x_i^{(i)} \neq 0$ and $x_{i+1}^{(i)}=x_{i+2}^{(i)}=\dots=x_n^{(i)}=0$.
    \item For all $1 \leq i \leq n$, $\|\mathbf{b}_i\| \leq \left\|\mathbf{b}_i +\sum_{j=1}^{i-1}x_j\mathbf{b}_j\right\|$ for all $x_j \in \mathbb{Z}$.
\end{enumerate}
\begin{theorem}
For each lattice $\Lambda$, there exists a basis $B$ for $\Lambda$ that is strongly reduced.
\end{theorem}
\begin{proof}
We will prove constructively that given any basis $B=\{\mathbf{b}_1,\dots,\mathbf{b}_n\}$ for $\Lambda$ then we move to a new basis $B$ that is strongly reduced using only unimodular operations. First, we construct the $\mathbf{v}_i=\sum_{i=1}^nx_j^{(i)}\mathbf{b}_j$ as in the first definition of strongly reduced bases, that is, so each $\mathbf{v}_i$ is linearly independent over $\mathbb{R}$ and $\|\mathbf{v}_i\|=\lambda_i$. Then for each $1 \leq i \leq n$, there exists a $1 \leq j \leq n$ such that $x_j^{(i)} \neq 0$ since by construction these vectors are linearly independent. Therefore we can freely reindex the basis vectors such that $x_i^{(i)} \neq 0$ for each $1 \leq i \leq n$. Suppose then that $x_j^{(i)} \neq 0$ for some $1 \leq i < j \leq n$, and suppose that this is the smallest such $i$. Then we can construct a new basis $B^\prime=\{\mathbf{b}_1^\prime,\dots,\mathbf{b}_n^\prime\}$ such that $\mathbf{b}_1^\prime=\mathbf{b}_1,\dots,\mathbf{b}_{i-1}^\prime=\mathbf{b}_{i-1}$ and $\mathbf{b}_i^\prime=\sum_{j=i}^ny_j^{(i)}\mathbf{b}_j$ where $y_j^{(i)}=\frac{x_j^{(i)}}{\gcd(x_i^{(i)},\dots,x_n^{(i)})}\in \mathbb{Z}$, by performing a suitable transform of the vectors $\mathbf{b}_{i+1},\dots,\mathbf{b}_n$. So, we can assume that the new basis $B^\prime$ satisfies the first property of strongly reduced basis. It follows easily that we can then replace any $\mathbf{b}_i^\prime$ with $\mathbf{b}_i^\dag=\mathbf{b}_i^\prime+\sum_{j=1}^{i-1}z_i \mathbf{b}_j^\prime$ for some $z_j \in \mathbb{Z}$ such that $\|\mathbf{b}_i^\dag\|\leq \left\|\mathbf{b}_i^\dag +\sum_{j=1}^{i-1}x_j\mathbf{b}_j^\dag\right\|$ for any $x_j \in \mathbb{Z}$ without violating the first property of strongly reduced bases. In this way, we can construct the strongly reduced basis.
\end{proof}
We now want to prove some useful properties about strongly reduced lattice bases.
\begin{proof}[Proof of Theorem \ref{successiveminima}.]
    First, construct $\mathbf{v}_i=\sum_{j=1}^n x_j^{(i)}\mathbf{b}_j$ as before for some $x_j^{(i)} \in \mathbb{Z}$ so that $\mathbf{v}_1,\dots,\mathbf{v}_n$ are linearly independent over $\mathbb{R}$ and $\|\mathbf{v}_i\|=\lambda_i$ for all $1 \leq i \leq n$. Since $B$ is strongly reduced, by the first property, for each $i$ we have $x_{i+1}^{(i)}=\dots=x_n^{(i)}=0$. By the second property of strongly reduced bases, if $|x_i^{(i)}|=1$ we have $\|\mathbf{b}_i\|=\lambda_i$, which clearly would satisfy the inequality stated in the theorem. Therefore, assume instead that $\|\mathbf{b}_i\|>\lambda_i$, so $|x_i^{(i)}| \geq 2$. Then
    \begin{align}
\lambda_i^2=\|\mathbf{v}_i\|^2=\sum_{j=1}^i\left(x_j^{(i)}+\sum_{k=j+1}^ix_k^{(i)}\mu_{j,k}\right)^2\|\mathbf{b}_j(j)\|^2 \geq {x_i^{(i)}}^2\|\mathbf{b}_i(i)\|^2 \geq 4\|\mathbf{b}_i(i)\|^2 \iff \|\mathbf{b}_i(i)\|^2 \leq \frac{1}{4}\lambda_i^2. \label{1}
    \end{align}
    Now define $\mathbf{w}_i=\mathbf{b}_i+\sum_{j=1}^{i-1}x_j\mathbf{b}_j$ such that 
    \begin{align*}
    &|\mu_{i,i-1}+x_{i-1}|\leq 1/2,
    \\&|\mu_{i,i-2}+x_{i-1}\mu_{i-1,i-2}+x_{i-2}| \leq 1/2,
    \\
    &\vdots
    \\& \left|x_1+\mu_{i,1}+\sum_{j=2}^{i-1}x_j\mu_{j,1}\right|\leq 1/2.
\end{align*}
By the second property of strongly reduced bases, $\|\mathbf{w}_i\|\geq \|\mathbf{b}_i\|$. Let $S_i$ denote the subset of $\{1,\dots,i-1\}$ such that $\|\mathbf{b}_j\|=\lambda_j$ for all $j \in S_i$ and $T_i$ denote the complement of this set. By definition of $k_i$, $|T_i|=k_i$. Then for any $j \in T_i$, we have $\|\mathbf{b}_j(j)\|^2 \leq \frac{1}{4}\lambda_j^2$ by a similar argument to \ref{1}. Therefore using the fact that $\lambda_j^2\leq \lambda_k^2$ for all $j\leq k$,
\begin{align*}
\|\mathbf{b}_i\|^2 \leq \|\mathbf{w}_i\|^2 &\leq \|\mathbf{b}_i(i)\|^2+\sum_{j=1}^{i-1}\frac{1}{4}\|\mathbf{b}_j(j)\|^2=\|\mathbf{b}_i(i)\|^2+\sum_{j \in S_i}\frac{1}{4}\|\mathbf{b}_j(j)\|^2+\sum_{j \in T_i}\frac{1}{4}\|\mathbf{b}_j(j)\|^2 \\&\leq \frac{1}{4}\lambda_i^2+\frac{1}{16}\sum_{j \in T_i}\lambda_j^2+\frac{1}{4}\sum_{j \in S_i}\lambda_j^2 \leq \left(\frac{i-k_i}{4}+\frac{k_i}{16}\right)\lambda_i^2,
\end{align*}
as required.
\end{proof}
It is significant to remark that if $B$ is strongly reduced, this means that it is guaranteed that $\|\mathbf{b}_i\|=\lambda_i$ for $i \in \{1,2,3,4\}$.

\begin{proof}[Proof of Theorem \ref{orthogonalitydefect}.]
    By Theorem \ref{successiveminima}, since $B$ is strongly reduced we have
    \begin{align*}
        \|\mathbf{b}_i\|^2 \leq \max\left\{1,\frac{i-k_i}{4}+\frac{k_i}{16}\right\}\lambda_i^2,
    \end{align*}
    where $k_i$ is defined as before. Therefore
    \begin{align*}
        \Delta(B) \leq \prod_{i=1}^n\max\left\{1,\frac{i-k_i}{4}+\frac{k_i}{16}\right\}\frac{\lambda_i^2}{\|\mathbf{b}_i(i)\|^2} \leq \gamma_n^n \prod_{i=1}^n\max\left\{1,\frac{i-k_i}{4}+\frac{k_i}{16}\right\},
    \end{align*}
    where we have used Minkowski's second theorem which states that $\prod_{i=1}^n \frac{\lambda_i^2}{\|\mathbf{b}_i(i)\|^2} \leq \gamma_n^n$. We are therefore interested in when the product
    \begin{align*}
        \prod_{i=1}^n\max\left\{1,\frac{i-k_i}{4}+\frac{k_i}{16}\right\}=\prod_{i=4}^n\left(\frac{i-k_i}{4}+\frac{k_i}{16}\right)
    \end{align*}
    is maximised (since $k_i \leq i-4$, since $\|\mathbf{b}_i\|=\lambda_i$ for $i \leq 4$). Suppose that $k$ vectors in $B$ satisfy $\|\mathbf{b}_i\|>\lambda_i$. It is easily seen that the product above is maximised when we have $\|\mathbf{b}_n\|>\lambda_n,\dots,\|\mathbf{b}_{n-k+1}\|>\lambda_{n-k+1}$, so we consider the product
    \begin{align*}
        f(n,k)=\prod_{i=0}^{k-1}\left(\frac{n-k+1}{4}+\frac{i}{16}\right).
    \end{align*}
    We want to know where the maximum of $f(n,k)$ occurs. Consider then
    \begin{align*}
        &\frac{f(n,k+1)}{f(n,k)}=\frac{\prod_{i=0}^k\left(\frac{n-k}{4}+\frac{i}{16}\right)}{\prod_{i=0}^{k-1}\left(\frac{n-k+1}{4}+\frac{i}{16}\right)}\\&=\frac{(4n-4k)(4n-4k+1)\dots(4n-4k+4)(4n-4k+5)\dots(4n-3k-1)(4n-3k)}{16(4n-4k+4)(4n-4k+5)\dots(4n-3k+1)(4n-3k+2)(4n-3k+3)}
        \\&=\frac{(4n-4k)(4n-4k+1)(4n-4k+2)(4n-4k+3)}{16(4n-3k+1)(4n-3k+2)(4n-3k+3)}.
    \end{align*}
    If the above expression is greater than $1$, then the product is increasing for increasing $k$, and conversely, decreasing in $k$ if the expression is less than $1$. Hence, if it is increasing then using the change in variable $\sigma=n-k$,
    \begin{align*}
        4\sigma(4\sigma+1)(4\sigma+2)(4\sigma+3) \geq 16(n+3\sigma+1)(n+3\sigma+2)(n+3\sigma+3),
    \end{align*}
    and otherwise if it is decreasing,
    \begin{align*}
        4\sigma(4\sigma+1)(4\sigma+2)(4\sigma+3) \leq 16(n+3\sigma+1)(n+3\sigma+2)(n+3\sigma+3).
    \end{align*}
    Since the expression on both sides of the inequalities has negative roots and are positive quartic/cubic polynomials, it holds that if $\frac{f(n,k+1)}{f(n,k)} \geq 1$ for a value of $k$ then $\frac{f(n,k^\prime+1)}{f(n,k^\prime)} \geq 1$ for all $k^\prime \leq k$ and likewise if $\frac{f(n,k+1)}{f(n,k)} \leq 1$ for some value of $k$ then $\frac{f(n,k^\prime+1)}{f(n,k^\prime)} \leq 1$ for all $k^\prime \geq k$. Hence, we want to find the integer $k$ for which the product is maximised, which will sit on either side of the maximal root of
    \begin{align*}
        y=4\sigma(4\sigma+1)(4\sigma+2)(4\sigma+3)- 16(n+3\sigma+1)(n+3\sigma+2)(n+3\sigma+3).
    \end{align*}
    This is a quartic polynomial in $\sigma$ which has a positive and negative root, but we are only interested in the positive root, which is $\beta_n$ exactly.
\end{proof}
We will finish by comparing this bound with the upper bound on the orthogonality defect of HKZ reduced lattices, which is given by
\begin{align*}
    \Delta(B) \leq \gamma_n^n\prod_{i=1}^n\frac{i+3}{4}.
\end{align*}
We will define by $f_H(n)=\prod_{i=1}^n\frac{i+3}{4}$ and $f_S(n)=\max_{\alpha_n}\prod_{j=\alpha_n}^n\left(\frac{n-\alpha_n+1}{4}+\frac{i}{16}\right)$, the respective upper bounds of the orthogonality defect of HKZ reduced bases \cite{HKZ} and strongly reduced bases, divided by $\gamma_n^n$. It is well known by Minkowski's second theorem that there always exists a lattice with a basis whose orthogonality defect achieves $\gamma_n^n$ exactly, and this orthogonality defect is minimal, so this is a good measure of the quality of the orthogonality defect of such a basis. We refer the reader to table 1 to compare the bounds between the two bases.
\begin{table}[]\label{table}
\centering
\begin{tabular}{|l|l|l|}
\hline
$n$  & $f_H(n)$               & $f_S(n)$               \\ \hline
$4$  & $3.281$                & $1$                    \\ \hline
$5$  & $6.563$                & $1.25$                 \\ \hline
$6$  & $14.766$               & $1.641$                \\ \hline
$7$  & $36.914$               & $2.344$                \\ \hline
$8$  & $101.514$              & $3.809$                \\ \hline
$9$  & $304.541$              & $6.427$                \\ \hline
$10$ & $989.758$              & $11.523$               \\ \hline
$15$ & $993779.193$           & $577.568$              \\ \hline
$20$ & $3.919 \times 10^9$    & $88919.111$            \\ \hline
$30$ & $1.255 \times 10^{18}$ & $2.786 \times 10^{10}$ \\ \hline
\end{tabular}
\caption{Bounds comparing $f_H(n)$ to $f_S(n)$ for each rank $n$.}
\end{table}
\section{Concluding Remarks}
In this work, we developed a new method of reduction of lattice bases that promises basis vectors with relatively short length, and whose orthogonality defect is also quite small. These bounds are compared to the bounds on other notions of strong reduction such as HKZ reduced lattices. It remains an open problem to develop a method of reducing lattice bases in this manner, however, as intuitively strongly reducing bases would require prior knowledge about the successive minima of the lattice, and so an algorithm to strongly reduce a lattice basis would require a lot of computation in order to enumerate all the successive minima of the lattice.

\end{document}